\documentclass[11pt,a4paper]{amsart}

\usepackage[T1]{fontenc}
\usepackage{lmodern}
\usepackage[margin=28mm]{geometry}
\usepackage{amsmath,amssymb,amsthm,mathtools}
\usepackage{microtype}
\usepackage{enumitem}
\usepackage{tikz}
\usepackage{graphicx}
\usetikzlibrary{arrows.meta,decorations.pathreplacing}
\definecolor{diagramaccent}{RGB}{29,96,104}
\usepackage[section]{placeins}
\usepackage[hidelinks]{hyperref}
\usepackage[nameinlink,noabbrev]{cleveref}
\hypersetup{pdftitle={Logarithmic edge-length ratios for planar drawings of 2-trees},
  pdfsubject={A logarithmic upper bound for the planar edge-length ratio of 2-trees}}
\setlist[enumerate]{leftmargin=*,itemsep=5pt,topsep=6pt}
\allowdisplaybreaks[2]
\newtheorem{theorem}{Theorem}[section]
\newtheorem{lemma}[theorem]{Lemma}
\newtheorem{proposition}[theorem]{Proposition}
\newtheorem{corollary}[theorem]{Corollary}
\theoremstyle{definition}
\newtheorem{definition}[theorem]{Definition}
\theoremstyle{remark}

\newcommand{\floor}[1]{\left\lfloor#1\right\rfloor}
\title[Edge-length ratios of $2$-trees]
{Logarithmic edge-length ratios for planar drawings of $2$-trees}

\author[Z. Sabaei]{Zahra Sabaei}
\address{Amirkabir University of Technology}
\email{zsabaei72@gmail.com}
\author[F. Emami]{Fatemeh Emami}
\address{Concordia University}
\email{fatima.emami13@gmail.com}
\author[Y. Tousinejad]{Yasmin Tousinejad}
\address{Department of Mathematics, Uppsala University}
\email{yasmin.tousinejad@math.uu.se}
\date{}

\begin{document}
\begin{abstract}
For a planar graph $G$ with at least one edge, let $\rho(G)$ be the
infimum of the ratio of the longest to the shortest edge length over
all planar straight-line drawings of $G$. We prove that every
$n$-vertex $2$-tree $G$ satisfies
$\rho(G)\le4+5\lfloor\log_2(2n-3)\rfloor$.
This improves the known $O(\sqrt n)$ upper bound and, together with
the lower bound of Bla\v{z}ej, Fiala and Liotta, shows that the maximum
of $\rho(G)$ over $n$-vertex $2$-trees is $\Theta(\log n)$.
The proof adapts Biedl's recursive bar construction using nonintegral
heights and applies a height-preserving straightening theorem.
\end{abstract}
\maketitle

\section{Introduction}
\label{sec:introduction}

Relations between edge lengths in graph drawings have been studied
since at least the late
1980s~\cite{BhattCosmadakis1987,EadesWormald1990}, and controlling
differences between these lengths is important both for the readability
of planar straight-line drawings~\cite{DiBattista1999} and for
applications such as unit disk graph approximation~\cite{Chen2011}.
However, requiring all edges to have the same length leads to a
difficult decision problem, as Eades and Wormald proved that deciding
whether a biconnected planar graph admits such a drawing is
NP-hard~\cite{EadesWormald1990}.

These considerations motivate the study of how small the ratio between
the longest and shortest edge lengths can be for particular graph
classes. Lazard, Lenhart and Liotta~\cite{LazardLenhartLiotta2019}
studied this ratio for outerplanar graphs and asked whether it is
bounded by a constant for $2$-trees. For $n$-vertex $2$-trees, Borrazzo
and Frati~\cite{BorrazzoFrati2020} proved an $O(n^{0.695})$ upper bound,
whereas Bla\v{z}ej, Fiala and Liotta~\cite{BFL2021} established a lower
bound of order $\log n$, thereby showing that no constant bound is
possible. More recently, Bekos et al.~\cite{Bekos2026} improved the upper
bound to $O(\sqrt n)$ through their work on weakly leveled planar
drawings and explicitly asked whether the remaining gap could be closed.

We close this gap by proving that every $n$-vertex $2$-tree admits a
planar straight-line drawing whose edge-length ratio is bounded by a
constant times $\log n$. The proof uses a recursive bar construction
in which the heights of adjacent bars differ by at least one unit,
while those of nonadjacent bars need not be separated by a fixed
positive amount, allowing the total height to remain logarithmic.
Applying a height-preserving straightening theorem and then compressing
the horizontal coordinate gives the required drawings, since the
vertical separation prevents short edges and the logarithmic total
height controls the longest edge.

\subsection{Notation and main results}

All graphs in this paper are finite and simple. For a graph $G$, let
$V(G)$ and $E(G)$ denote its vertex set and edge set, respectively, and
let $n\coloneqq|V(G)|$. Let $N(v)$ denote the set of neighbors of a
vertex $v$ in the graph under consideration. A \emph{$2$-tree} is
obtained from a single edge by repeatedly adding a new vertex whose
only neighbors are the two endpoints of an existing edge.
Thus, for a $2$-tree
\begin{equation}
  |E(G)|=1+2(n-2)=2n-3.
  \label{eq:edge-count}
\end{equation}

A \emph{planar straight-line drawing} assigns distinct points of
$\mathbb R^2$ to the vertices and represents each edge by the segment
between its endpoints. No edge contains a nonincident vertex, and two edges
intersect only at a common endpoint. We allow arbitrary real coordinates
and do not fix a planar embedding in advance.

Let $G$ be a planar graph with at least one edge. For a planar
straight-line drawing $D$ of $G$ and an edge $e\in E(G)$, let
$|e|_D$ denote the Euclidean length of $e$. The edge-length ratio
of $D$ and the planar edge-length ratio of $G$ are, respectively,
\begin{equation}
  r(D)\coloneqq
  \frac{\max\{|e|_D:e\in E(G)\}}
       {\min\{|e|_D:e\in E(G)\}},
  \qquad
  \rho(G)\coloneqq\inf_D r(D),
  \label{eq:ratio}
\end{equation}
where the infimum ranges over all planar straight-line drawings of $G$.
Using an infimum avoids requiring a drawing with ratio $\rho(G)$.
For $n\ge2$, let $R_2(n)$ denote the maximum planar edge-length ratio among all $n$-vertex $2$-trees:
\begin{equation}
  R_2(n)\coloneqq
  \max\{\rho(G):G\text{ is a $2$-tree with $n$ vertices}\}.
  \label{eq:extremal}
\end{equation}
There are finitely many such graphs up to isomorphism, so the maximum in
\eqref{eq:extremal} exists.

\begin{theorem}
\label{thm:main}
For every $2$-tree $G$ with $n\ge2$ vertices,
\begin{equation}
  \rho(G)\le4+5\floor{\log_2(2n-3)}.
  \label{eq:main-upper}
\end{equation}
Equivalently, for every $\delta>0$, there is a planar straight-line
drawing of $G$ whose edge-length ratio is less than
$4+5\lfloor\log_2(2n-3)\rfloor+\delta$.
\end{theorem}

The lower bound of Bla\v{z}ej, Fiala and Liotta~\cite{BFL2021}, together
with \cref{thm:main}, gives the following consequence. A short deduction
of the explicit lower bound is included in \cref{sec:consequences}.

\begin{corollary}
\label[corollary]{cor:order}
For every integer $n\ge3$,
\begin{equation}
  \floor{\frac16\log_{10} \left(\frac{3n-4}{5}\right)}
  \le R_2(n)
  \le4+5\floor{\log_2(2n-3)}.
  \label{eq:two-bounds}
\end{equation}
Consequently, $R_2(n)=\Theta(\log n)$.
\end{corollary}

The lower estimate in \eqref{eq:two-bounds} concerns the largest ratio among
graphs of a given order; it is not a lower bound for every $2$-tree.
The coefficient $5$ in \eqref{eq:main-upper} comes from the construction,
and no claim is made that it is optimal.

\subsection{Proof overview}

Biedl~\cite[Section~3]{Biedl2011} develops recursive bar constructions
for small drawings of series-parallel graphs. Our proof uses her arrangement of
series and parallel components, including the reversed order of the smaller
components in the series step. The change is in the choice of heights and
the conditions at the two terminal bars. We use only vertical edge segments
and allow nonintegral bar heights. In each height recurrence, the
contribution of a largest component has no additive term. Every
component whose contribution has an additive term contains at most
half the edges of the component being constructed. Several nonadjacent
bars can then be placed at distinct rational heights within one unit
of vertical space.

In \cref{sec:conversion}, the height-preserving straightening theorem
of Pach and T\'oth~\cite{PachToth2004} reduces the problem to constructing
bars of bounded total height with at least unit separation along every
edge. We give this construction and prove \cref{thm:main} in
\cref{sec:construction}. In \cref{sec:consequences}, we deduce
\cref{cor:order}, extend the upper bound to graphs with treewidth at
most $2$, and explain the role of nonintegral heights.

\section{From bars to straight-line drawings}
\label{sec:conversion}

\begin{definition}
A \emph{bar representation} assigns a horizontal closed segment $B_v$ to
each vertex $v$ and a vertical segment to each edge $uv$, joining $B_u$ to
$B_v$. The bars are pairwise disjoint, no edge segment meets a nonincident
bar, and edge segments have pairwise disjoint interiors. Only the specified
edge segments represent edges; other vertical connections between bars
need not represent edges of the graph.
\end{definition}

\begin{lemma}
\label[lemma]{lem:conversion}
Suppose that a graph $G$ with at least one edge has a bar representation
whose bar heights belong to $[0,H]$, and that the endpoint heights of every
edge differ by at least $1$. Then $\rho(G)\le H$. Moreover, for every
$\delta>0$, there is a planar straight-line drawing of $G$ with ratio less
than $H+\delta$.
\end{lemma}

\begin{proof}
Surround the bars by pairwise disjoint thin rectangles, each extending
slightly beyond the endpoints of its bar. Since the representation is finite,
the rectangles can avoid all nonincident edge segments, and the two
rectangles incident with any edge can be vertically separated. Truncate
each vertical edge at the boundaries of its endpoint rectangles.

In each rectangle, choose an interior point to represent its vertex and
join it by straight segments to the endpoints of the truncated incident
edges. The endpoints on the top side of a rectangle are distinct, since
otherwise two original edge interiors would overlap; the same holds on the
bottom side. Segments joining the chosen point to the top side intersect
only at that point, and the same is true for the bottom side. The two
groups lie on opposite sides of the horizontal line through the point.
The resulting drawing is therefore planar, and every edge is strictly
$y$-monotone.

Fix $0<\varepsilon<1/2$ and use the openness of the rectangles to choose
interior points at pairwise distinct heights, each differing by less
than $\varepsilon$ from the height of its bar. The resulting vertex
heights lie in $[-\varepsilon,H+\varepsilon]$, while the endpoint heights
of every edge differ by at least $1-2\varepsilon$. Since the chosen
heights are pairwise distinct, a rotation by $90$ degrees allows us to
apply Theorem~2 of Pach and T\'oth~\cite{PachToth2004}; rotating the
resulting drawing back then gives a planar straight-line drawing
with the same vertex heights.

Compress the horizontal coordinate by a positive factor so that the width
is at most $\varepsilon$. This invertible linear map preserves planarity.
For the resulting drawing $D_\varepsilon$,
\begin{equation}
  1-2\varepsilon\le |e|_{D_\varepsilon}
  \le\sqrt{(H+2\varepsilon)^2+\varepsilon^2}
  \qquad(e\in E(G)).
  \label{eq:edge-estimates}
\end{equation}
Consequently, \eqref{eq:ratio} and \eqref{eq:edge-estimates} give
\begin{equation}
  \rho(G)\le r(D_\varepsilon)
  \le\frac{\sqrt{(H+2\varepsilon)^2+\varepsilon^2}}
           {1-2\varepsilon},
  \qquad
  \lim_{\varepsilon\downarrow0}
  \frac{\sqrt{(H+2\varepsilon)^2+\varepsilon^2}}
       {1-2\varepsilon}=H.
  \label{eq:compression-limit}
\end{equation}
Both conclusions follow from \eqref{eq:compression-limit}.
\end{proof}

An example of the conversion is given in \cref{fig:conversion}.

\begin{figure}[htbp]
  \centering
  \includegraphics[width=0.9\linewidth]{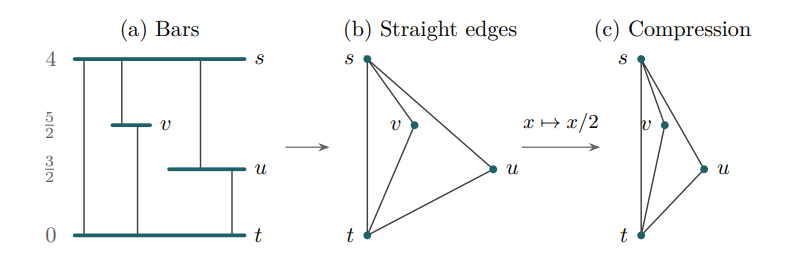}
 \caption{The conversion for the $2$-tree with vertices $s,t,u,v$ and edges
$st,su,tu,sv,tv$.
The straight-line drawing keeps the four bar heights, and horizontal
compression keeps every vertical difference along an edge. All three
representations are planar. The factor $1/2$ in (c) is illustrative;
the width can be made arbitrarily small as in \cref{lem:conversion}.}
\label{fig:conversion}
\end{figure}

\section{The recursive construction}
\label{sec:construction}

\subsection{Components and their terminals}
\label{subsec:components}

Fix a construction sequence of the $2$-tree. Each vertex other than the
initial two vertices is attached to one existing edge, and the two edges
created at that step are recorded as children of that edge. For each edge
$ab$, let $F_{ab}$ consist of $ab$ and all edges obtained from it by repeated
child steps, together with their endpoints. Its \emph{terminals} are $a,b$.
If $x_1,\ldots,x_k$ are the vertices attached directly to $ab$, then
\begin{equation}
  F_{ab}=ab\;\cup\;\bigcup_{i=1}^{k}
  \bigl(F_{ax_i}\cup F_{x_ib}\bigr),
  \label{eq:decomposition}
\end{equation}
where $ab$ on the right denotes the graph consisting of that one edge.
If $k=0$, the union is empty and $F_{ab}=ab$.

We call $F_{ab}$ an \emph{edge component}, and call each
$F_{ax_i}\cup F_{x_ib}$ in \eqref{eq:decomposition} a \emph{series
component}, with terminals $a,b$. Its two edge components meet only at
$x_i$. Different series components in \eqref{eq:decomposition} meet only
at $a,b$, so the right side of that equation is a \emph{parallel
composition} of the edge $ab$ and those series components. All their edge
sets are disjoint. These facts follow from the construction sequence, since
every new vertex has only the two endpoints of its attachment edge as older
neighbors. The edge component of the initial edge is the whole graph.

For any component $C$, define $m(C)\coloneqq|E(C)|$. We construct bar
representations of edge components and series components simultaneously by
induction on $m(C)$. Terminals are ordered, with $s$ denoting the upper
terminal and $t$ the lower terminal. Vertical reflection allows either
ordering of a given pair of terminals.

\subsection{The boundary conditions}
\label{subsec:conditions}

For each bar $B_v$, let $y_v$ denote its height and let $\ell_v,r_v$
denote the horizontal coordinates of its left and right endpoints,
respectively, so that
\begin{equation}
  B_v=[\ell_v,r_v]\times\{y_v\}.
  \label{eq:bar}
\end{equation}
Let $c_{uv}$ be the horizontal coordinate of the vertical segment for the
edge $uv$. The representation lies in a rectangle $[0,W]\times[0,H]$,
where $W>0$ and $H\ge4$. We maintain the following conditions.

\begin{enumerate}[label=\textnormal{(B\arabic*)},ref=\textnormal{(B\arabic*)}]
\item\label{cond:proper}
The representation satisfies the definition in \cref{sec:conversion}, has
bars at both heights $0,H$, and every edge has endpoint-height difference
at least $1$. All edge columns are pairwise distinct.
\item\label{cond:right}
The bars of $s,t$ end at $W$, and every other bar ends strictly before $W$.
\item\label{cond:bands}
The terminal heights satisfy $y_s\in\{H-1,H\}$ and $y_t\in\{0,1\}$.
No bar has height in either $(0,1)$ or $(H-1,H)$.
\item\label{cond:bottom}
If $y_t=1$, then $t$ has a unique neighbor $x$ below it, at height $0$.
For $c\coloneqq c_{tx}$,
\begin{equation}
  r_x=c,\qquad
  r_v\le c\ \text{if }y_v=0,\qquad
  c_{tu}>c\ \text{if }u\in N(t)\setminus\{x\}.
  \label{eq:lower-condition}
\end{equation}
\item\label{cond:top}
If $y_s=H-1$, then $s$ has a unique neighbor $x$ above it, at height $H$.
For $c\coloneqq c_{sx}$,
\begin{equation}
  r_x=c,\qquad
  r_v\le c\ \text{if }y_v=H,\qquad
  c_{su}>c\ \text{if }u\in N(s)\setminus\{x\}.
  \label{eq:upper-condition}
\end{equation}
\end{enumerate}

Allowing a terminal to be one unit inside the boundary makes it possible
to add an edge in series without increasing the height. Conditions
\ref{cond:bottom} and \ref{cond:top} allow that terminal to be moved outside
the rectangle later. Positive horizontal scaling and horizontal translation preserve
all five conditions, and vertical reflection exchanges the terminal roles.
Throughout the construction, the horizontal intervals of different child
rectangles are separated by positive gaps.

\begin{lemma}[Increasing the height]
\label[lemma]{lem:enlarge}
A representation satisfying \ref{cond:proper}--\ref{cond:top} at height
$H$ can be enlarged to any height $H'\ge H$ while preserving these
conditions, without decreasing any endpoint-height difference
along an edge.
\end{lemma}

\begin{proof}
Choose $q\in(1,H-1)$ different from every bar height. Move each bar
above $q$ upward by $H'-H$, and extend the edges crossing the line
of height $q$. The inserted horizontal strip contains no bars,
so incidences and disjointness are preserved, while every
endpoint-height difference stays the same or increases. Bars in
the lower boundary band remain fixed, whereas those in the upper
boundary band move upward by $H'-H$. The horizontal coordinates
of bar endpoints and edge columns remain unchanged, so all five
conditions are preserved.
\end{proof}

\begin{lemma}[Moving the terminals to the boundary]
\label[lemma]{lem:release}
From a representation satisfying \ref{cond:proper}--\ref{cond:top} at height
$H$, one obtains a bar representation of height $H+2$ in which the terminal
bars span the full width at heights $0,H+2$. Every other bar has height
in $[1,H+1]$. Every edge has endpoint-height difference at least $1$, and
the edge columns remain distinct.
\end{lemma}

\begin{proof}
In the original coordinates, move $t$ to height $-1$ and $s$ to height
$H+1$, extend their bars across $[0,W]$, and reconnect incident edges in
their original columns. We check that this gives a valid new representation;
no continuous movement between the two representations is required.

If $y_t=0$, extending its incident edges downward introduces no intersection
because there are no bars below height $0$. Suppose instead that $y_t=1$,
and let $x,c$ be as in \eqref{eq:lower-condition}. For every neighbor
$u\ne x$, the new part of the segment for $tu$ crosses height $0$ strictly
to the right of all bars at that height. It crosses no bar in $(0,1)$ by
\ref{cond:bands}. The edge $tx$ is replaced by the segment in column $c$
between heights $-1$ and $0$, so its new length is $1$. Thus the
lower-terminal change creates no unwanted edge--bar intersection. The
upper-terminal change is justified in the same way by
\eqref{eq:upper-condition}.

The two changes are compatible. Since $H\ge4$, an exceptional neighbor
at height $H$ cannot be $t$, and an exceptional neighbor at height $0$
cannot be $s$. If $st$ is an edge, it is among the other incident edges
in each applicable boundary condition, so both extensions of that edge
are protected by the strict column inequalities. Distinct edge columns
exclude intersections between edge interiors, and the new terminal bars
lie outside all other bars. Finally translate upward by $1$. The two
terminal heights become $0,H+2$, and all other heights belong to
$[1,H+1]$.
\end{proof}

The lower-terminal change is illustrated in \cref{fig:release}.
\begin{figure}[htbp]
  \centering
  \includegraphics[width=0.9\linewidth]{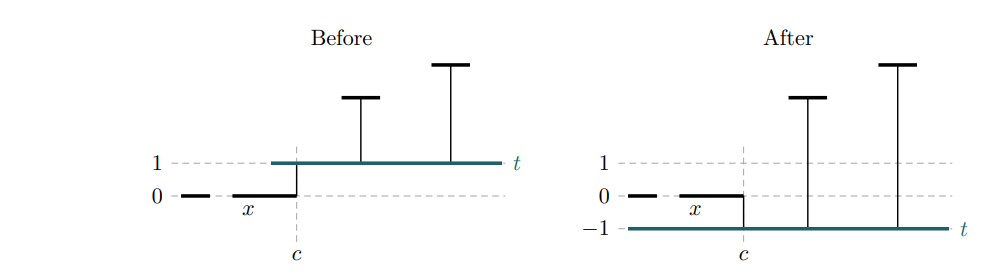}
\caption{The local lower boundary in \cref{lem:release}, before the
final upward translation. Every height-$0$ bar ends at or before $c$,
and the bar of $x$ ends at $c$. Every edge incident with $t$
other than $tx$ lies in a column strictly to the right of $c$.
The bar of $t$ can therefore be placed at height $-1$, with its
incident edges reconnected in their original columns.
The upper boundary is omitted.}
\label{fig:release}
\end{figure}

After \cref{lem:release}, the distance between the terminal bars can
be increased by inserting a horizontal strip at a height containing
no bar and extending the edges that cross it. Every nonterminal
bar remains at least one unit from both the upper and lower
boundaries, allowing the component to fit between any two terminal
heights with sufficiently large separation.

\subsection{The initial edge and parallel composition}
\label{subsec:parallel}

For one edge, use two bars spanning $[0,1]$ at heights $0,4$ and an edge
in column $1/2$. All five conditions hold.

Consider an edge component with at least one attached vertex. Its parallel
branches are the direct edge and the series components in
\eqref{eq:decomposition}. Order these branches as $C_1,\ldots,C_p$ so that
$m(C_1)\ge\cdots\ge m(C_p)$, where $p\ge2$. Let $H_i$ be the height
of the recursively constructed representation of $C_i$, and define
\begin{equation}
  H\coloneqq\max\{H_1,H_2+4,\ldots,H_p+4\}.
  \label{eq:parallel-height}
\end{equation}
Increase the height of $C_1$ to $H$ by \cref{lem:enlarge}.
Denote its terminals by $s,t$.

For every $i\ge2$, apply \cref{lem:release} to $C_i$. By
\eqref{eq:parallel-height} and \ref{cond:bands},
\begin{equation}
  H_i+2\le H-2\le y_s-y_t.
  \label{eq:parallel-fit}
\end{equation}
Increase its terminal separation to $y_s-y_t$, place it in a disjoint
horizontal interval to the right of $C_1$, and identify its lower and upper
terminal bars with those of $t,s$, respectively. Extend the resulting two
terminal bars to a new right boundary beyond all component rectangles.

Every added nonterminal bar has height in $[y_t+1,y_s-1]$, so the extended
terminal bars meet no added nonterminal bar or nonincident edge. Inside
each new rectangle they coincide with its original terminal bars, and
between rectangles they pass through empty gaps. Different rectangles have
disjoint edge columns. Thus \ref{cond:proper} and \ref{cond:right} hold.

No added bar enters $(0,1)$ or $(H-1,H)$, and the terminal heights
do not change, proving \ref{cond:bands}. If $y_t=1$, no added bar is below
$t$ and every added edge incident with $t$ has its column strictly to the
right of the rectangle for $C_1$. Therefore \eqref{eq:lower-condition}
remains valid. The argument for \eqref{eq:upper-condition} is identical.
This proves all five conditions for the parallel composition, including
its direct terminal edge, which occurs once. The case with both
terminals on the boundary is shown in \cref{fig:compositions}(a).

\subsection{Series composition}
\label{subsec:series}

Let $C=F_{sx}\cup F_{xt}$ be a series component with terminals $s,t$.
By exchanging the terminal names if necessary, assume
$m(F_{sx})\ge m(F_{xt})$. We keep the recursively constructed
representation of $F_{sx}$ with $s$ above $x$. Denote its height by $H_0$.
The other terminal ordering follows by vertical reflection.

\subsubsection{The smaller component is one edge}
\label{subsec:series-edge}

Suppose that $F_{xt}=xt$, and let $b\coloneqq y_x\in\{0,1\}$.
Keep the height $H_0$ and set $y_t\coloneqq1-b$. Choose a new column
$c>W$ and a new right boundary $W'>c$. Extend the bar of $x$ to end at
$c$, introduce
\begin{equation}
  B_t\coloneqq[c,W']\times\{1-b\},
  \label{eq:new-terminal}
\end{equation}
and represent $xt$ by the vertical segment in column $c$. Extend the bar
of $s$ to $W'$.

The new edge has length $1$, and all new parts lie to the right of the
rectangle for $F_{sx}$, so \ref{cond:proper}--\ref{cond:bands} hold. If
$y_t=1$, its only neighbor is $x$ at height $0$, whose bar ends at $c$;
all other height-$0$ bars end before $c$. Thus
\eqref{eq:lower-condition} holds. If $y_t=0$, no lower condition is
required. The upper condition remains valid because $x$ is below height
$H_0-1$, and neither its extension nor the new edge changes any bar at
height $H_0$ or any edge incident with $s$.

The case $y_x=0$ is illustrated in \cref{fig:compositions}(b).
\begin{figure}[htbp]
  \centering
  \includegraphics[width=0.9\linewidth]{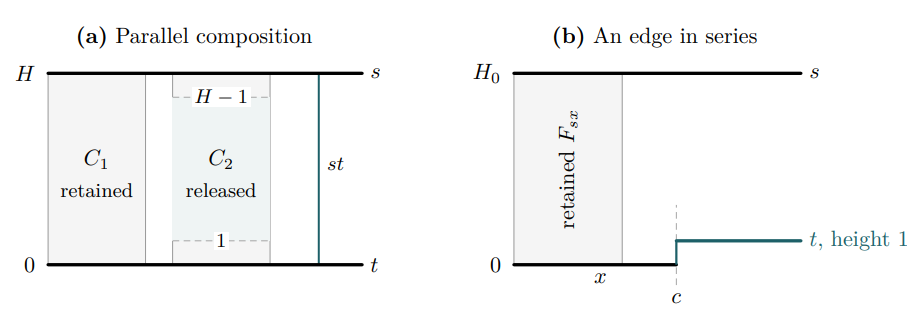}
 \caption{Two composition steps. In (a), the corresponding terminal bars of a largest branch and a released branch are merged, and the direct edge $st$ occupies a
separate column. The shaded middle region of $C_2$ contains its nonterminal
bars, with unit margins at the top and bottom. The case $y_s=H,y_t=0$
is shown. In (b), an edge is added in series when $y_x=0$, placing $t$ at
height $1$ without increasing the height $H_0$. Shaded rectangles represent
components; their internal bars and edges are omitted.}
\label{fig:compositions}
\end{figure}

\subsubsection{The smaller component has attached vertices}
\label{subsec:series-many}

Suppose now that $F_{xt}\ne xt$. Let $u_1,\ldots,u_k$ be the vertices
attached directly to $xt$, where $k\ge1$, and define
\begin{equation}
  A_i\coloneqq F_{xu_i},\qquad B_i\coloneqq F_{u_it}
  \quad(1\le i\le k).
  \label{eq:small-components}
\end{equation}
By \eqref{eq:decomposition},
\begin{equation}
  F_{xt}=xt\;\cup\;\bigcup_{i=1}^{k}(A_i\cup B_i).
  \label{eq:smaller-decomposition}
\end{equation}
The vertices $u_i$ are pairwise nonadjacent, since each was attached to
$xt$ and no step adds an edge between two existing vertices.

Recursively construct bar representations of all $A_i,B_i$,
denote their heights by $H(A_i),H(B_i)$, and define
\begin{equation}
  H\coloneqq\max\left\{H_0,
  5+\max_{1\le i\le k}\{H(A_i),H(B_i)\}\right\}.
  \label{eq:series-height}
\end{equation}
Increase the representation of $F_{sx}$ to height $H$. Let
$b\coloneqq y_x\in\{0,1\}$ and choose
\begin{equation}
  y_t\coloneqq1-b,\qquad
  z_i\coloneqq H-1-\frac{i}{k+1},\qquad
  y_{u_i}\coloneqq z_i\quad(1\le i\le k).
  \label{eq:apex-heights}
\end{equation}
Then
\begin{equation}
  H-2<z_k<\cdots<z_1<H-1.
  \label{eq:apex-order}
\end{equation}
Apply \cref{lem:release} to each $A_i,B_i$, orienting $x,t$ below $u_i$.
For $b'\in\{0,1\}$, \eqref{eq:series-height} and
\eqref{eq:apex-order} give
\begin{equation}
  z_i-b'>H-3
  \ge\max\{H(A_i),H(B_i)\}+2.
  \label{eq:series-fit}
\end{equation}
Thus each released component can be enlarged to the required terminal
separation. Place $A_i$ between heights $b,z_i$ and $B_i$ between heights
$1-b,z_i$.

Choose disjoint horizontal intervals in the order
\begin{equation}
  F_{sx},\ A_1,\ldots,A_k,\ \text{gap},\ B_k,\ldots,B_1.
  \label{eq:box-order}
\end{equation}
Choose a column $c$ in the central gap. Merge the copies of the bar of
$x$ in $F_{sx},A_1,\ldots,A_k$, extending the resulting bar to end at
$c$. Merge the copies of the bar of $t$ in $B_k,\ldots,B_1$, extending
the result leftward to $c$. Represent $xt$ by the vertical segment in
column $c$ between heights $0,1$. For each $i$, join the two copies of
the bar of $u_i$ across the interval between $A_i$ and $B_i$. Finally
extend the bars of $s,t$ to a new right boundary beyond all the rectangles.
The arrangement is illustrated in \cref{fig:series}.
\begin{figure}[htbp]
  \centering
  \includegraphics[width=0.9\linewidth]{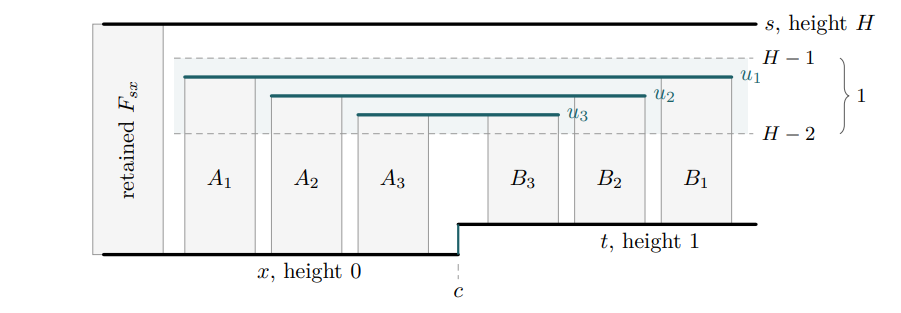}
  \caption{The remaining series step for $k=3$, $y_x=0$ and $y_s=H$.
The order of the $B_i$ is reversed, so the bar of $u_i$ passes above
every rectangle between $A_i$ and $B_i$. The highlighted interval has
height $1$ and contains all three heights $z_i$ from
\eqref{eq:apex-heights}. Its vertical scale is enlarged for clarity.
Each shaded rectangle represents a whole component, with internal bars
and edges omitted; the segment in column $c$ represents the new edge $xt$.}
\label{fig:series}
\end{figure}

We check the newly extended bars and the new edge. By
\eqref{eq:box-order}, the bar of $u_i$ passes between its own rectangles
only over rectangles $A_j,B_j$ with $j>i$. Their top heights satisfy
$z_j<z_i$ by \eqref{eq:apex-order}. All rectangles with $j<i$, as well
as the rectangle for $F_{sx}$, lie outside this extension. Hence the extended
bar meets no nonincident bar or edge. The extended bars of different
$u_i$ have distinct heights.

The bar of $x$ crosses the new rectangles only along the lower terminal
bars of the $A_i$. Every other bar there is at least one unit higher.
The corresponding statement holds for $t$ and the $B_i$. A bar inside
an $A_i$ may be at the same height as $t$, but it lies strictly left of
$c$, where the bar of $t$ does not extend. Similarly, a bar inside a
$B_i$ at the height of $x$ lies strictly right of $c$. Between
rectangles, the extended bars cross empty gaps. The edge $xt$ lies in a fresh
column between heights $0,1$, whereas all bars crossing the central gap
at higher levels have heights greater than $H-2$. Since
\eqref{eq:series-height} gives $H\ge9$, these bars do not meet $xt$.
The height of $s$ is at least $H-1$, and every new rectangle ends below
$H-1$, so extending the bar of $s$ is also valid.

Every edge inherited from a child still has endpoint-height
difference at least $1$, while the new edge $xt$ has difference $1$. All columns remain
distinct because child rectangles have disjoint horizontal intervals and
column $c$ is new. The rectangle for $F_{sx}$ still has bars at heights $0,H$.
This proves \ref{cond:proper}. Only $s,t$ reach the new right boundary,
proving \ref{cond:right}.

For \ref{cond:bands}, the terminal heights have the required values by
\eqref{eq:apex-heights}. A nonterminal bar inside a released child lies
at least one unit above its lower terminal and at least one unit below
its upper terminal. Together with \eqref{eq:apex-order}, this excludes new
bar heights in $(0,1)$ and $(H-1,H)$.

If $y_t=1$, then $y_x=0$. Every height-$0$ bar belongs to the
representation of $F_{sx}$ or is part of the merged bar of $x$,
so it ends at or before $c$. The bar of $x$ ends at $c$. All other edges incident with
$t$ belong to the $B_i$, have columns strictly greater than $c$, and join
$t$ to bars above it. Therefore \eqref{eq:lower-condition} holds. If
$y_t=0$, no lower condition is required. Finally, no new bar is at
height $H$, and no new edge is incident with $s$. Thus
\eqref{eq:upper-condition} remains valid, completing the construction.

\subsection{The height bound}

\begin{proposition}
\label[proposition]{prop:height}
Every edge component and every series component with $m\ge1$ edges has
a representation satisfying \ref{cond:proper}--\ref{cond:top} with height
at most $4+5\lfloor\log_2m\rfloor$.
\end{proposition}

\begin{proof}
For positive integers $m$, define
\begin{equation}
  f(m)\coloneqq4+5\floor{\log_2m}.
  \label{eq:height-function}
\end{equation}
We use strong induction on $m$. The one-edge representation has height
$4=f(1)$. In each recursive step, the chosen largest component has fewer
than $m$ edges and contributes its height without an additive term.

For a parallel composition, every branch other than $C_1$
has at most $\lfloor m/2\rfloor$ edges, and its
contribution to the maximum in \eqref{eq:parallel-height} has an
additive term of $4$. In a series step with one edge, the height
does not increase. In the remaining series step, let
$m_1\coloneqq m(F_{sx})$ and $m_2\coloneqq m(F_{xt})$. Then
\begin{equation}
  m_1+m_2=m,\qquad m_2\le m_1,\qquad
  m(A_i),m(B_i)\le m_2\le\floor{m/2}.
  \label{eq:size-bound}
\end{equation}
Thus every term of the form $H(A_i)+5$ or $H(B_i)+5$ in
\eqref{eq:series-height} comes from a component with at most
$\lfloor m/2\rfloor$ edges.

For $m\ge2$, let $q\coloneqq\lfloor\log_2m\rfloor$. Then
\begin{align}
  2^q\le m<2^{q+1}
  &\quad\Longrightarrow\quad
  2^{q-1}\le\floor{m/2}<2^q,\notag\\
  \floor{\log_2 \left(\floor{m/2}\right)}&=q-1,\notag\\
  f(\floor{m/2})+5&=4+5(q-1)+5=f(m).
  \label{eq:halving}
\end{align}
The function $f$ is nondecreasing. By the induction hypothesis and
\eqref{eq:parallel-height}, \eqref{eq:series-height},
\eqref{eq:size-bound} and \eqref{eq:halving}, every resulting height is
at most
\begin{equation}
  \max\{f(m-1),f(\floor{m/2})+5\}=f(m).
  \label{eq:height-induction}
\end{equation}
This completes the induction.
\end{proof}

\begin{proof}[Proof of \cref{thm:main}]
The initial edge component is $G$ itself. By \eqref{eq:edge-count} and
\cref{prop:height}, it has a bar representation of height at most
$4+5\lfloor\log_2(2n-3)\rfloor$, with endpoint-height difference at least
$1$ along every edge. Both conclusions of \cref{thm:main} follow from \cref{lem:conversion}.
\end{proof}

\section{Consequences and the role of the heights}
\label{sec:consequences}

We first deduce the lower bound from the known construction in
\cite{BFL2021}, without repeating its geometric proof. Let $G_0$ be a
triangle, whose three edges are created at stage $0$.
At each stage $h\ge1$, add five vertices to every edge created
at stage $h-1$, each adjacent to both endpoints of that edge. Then
\begin{equation}
  |V(G_h)|=3+15\sum_{j=0}^{h-1}10^j
  =\frac{5\cdot10^h+4}{3}.
  \label{eq:lower-size}
\end{equation}
For an integer $r\ge1$, let $k\coloneqq\lceil2+2\log_2r\rceil$.
The proof of Theorem~1 in the conference version of
\cite{BFL2021} gives
\begin{equation}
  \rho(G_{k+4r})\ge r\qquad(r\ge2).
  \label{eq:known-lower}
\end{equation}

If $F$ is a subgraph of $G$ and has at least one edge, restricting a drawing
to $F$ gives
\begin{equation}
  r(D|_F)\le r(D),\qquad \rho(F)\le\rho(G).
  \label{eq:monotonicity}
\end{equation}
The inequality $r\le2^{r-1}$ holds for integers $r\ge1$ by induction,
since $r+1\le2r\le2^r$. Therefore
\begin{equation}
  k\le\lceil2+2(r-1)\rceil=2r,
  \qquad k+4r\le6r,
  \qquad \rho(G_{6r})\ge r,
  \label{eq:depth-lower}
\end{equation}
where the last inequality follows from \eqref{eq:known-lower} and
\eqref{eq:monotonicity} for $r\ge2$, and from $\rho(G_6)\ge1$ for $r=1$.

\begin{proof}[Proof of \cref{cor:order}]
For $n\ge3$, define
\begin{equation}
  r\coloneqq\floor{\frac16\log_{10} \left(\frac{3n-4}{5}\right)}.
  \label{eq:choose-r}
\end{equation}
If $r=0$, the lower bound follows from $R_2(n)\ge1$. If $r\ge1$, then
\eqref{eq:lower-size} and \eqref{eq:choose-r} give
\begin{equation}
  |V(G_{6r})|=\frac{5\cdot10^{6r}+4}{3}\le n.
  \label{eq:fit-order}
\end{equation}
Repeatedly add a new vertex whose only neighbors are the two endpoints of an existing edge until the graph has $n$ vertices.
The result is a $2$-tree containing $G_{6r}$, so
\eqref{eq:monotonicity} and \eqref{eq:depth-lower} imply $R_2(n)\ge r$.
The upper bound is \cref{thm:main}, proving \eqref{eq:two-bounds}.
The two bounds have logarithmic order, so $R_2(n)=\Theta(\log n)$.
\end{proof}

\begin{corollary}
\label[corollary]{cor:partial}
Let $G$ be a graph with $n\ge3$ vertices, at least one edge, and treewidth at most $2$. Then
\[
  \rho(G)\le4+5\floor{\log_2(2n-3)}.
\]
\end{corollary}

\begin{proof}
Complete $G$ to a $2$-tree on the same vertex set, and apply
\cref{thm:main} and \eqref{eq:monotonicity}.
\end{proof}

The construction controls the geometric height differences along edges,
but does not give a corresponding bound on the number of distinct levels
between their endpoints. This distinction is illustrated by
\eqref{eq:apex-heights}, where arbitrarily many pairwise nonadjacent
vertices $u_i$ are placed at distinct heights within an interval of length
less than $1$. Since their bars overlap in horizontal projection, these
heights cannot be identified without making the bars intersect, while
replacing all distinct bar heights by consecutive integer levels preserves
their order but need not give logarithmic edge span. Indeed, in the series
construction, all $k$ heights $z_i$ lie strictly between those of $x$ and
$s$, so this replacement gives the edge $sx$ span at least $k+1$.
The logarithmic bound on geometric height therefore does not establish
a corresponding bound on edge span in a weakly leveled drawing,
in agreement with the lower bounds in~\cite{Bekos2026}.

\noindent
\\
{\bf Acknowledgements.}
Yasmin Tousinejad was partially supported by 
the Wallenberg AI, Autonomous Systems and Software Program (WASP) funded by the Knut and Alice Wallenberg Foundation.


\begin{thebibliography}{15}

\bibitem{Bekos2026}
M.~A. Bekos, G.~Da Lozzo, F.~Frati, S.~Gupta, P.~Kindermann,
G.~Liotta, I.~Rutter and I.~G. Tollis,
\emph{Weakly leveled planarity with bounded span},
Theoretical Computer Science \textbf{1069} (2026), article 115811.
\href{https://doi.org/10.1016/j.tcs.2026.115811}
{doi:10.1016/j.tcs.2026.115811}.
Preprint \href{https://arxiv.org/abs/2409.01889v2}{arXiv:2409.01889v2}
(2024), Corollary~44 and Section~6.

\bibitem{BhattCosmadakis1987}
S.~N. Bhatt and S.~S. Cosmadakis,
\emph{The complexity of minimizing wire lengths in {VLSI} layouts},
Information Processing Letters \textbf{25}(4) (1987), 263--267.
\href{https://doi.org/10.1016/0020-0190(87)90173-6}
{doi:10.1016/0020-0190(87)90173-6}.

\bibitem{Biedl2011}
T.~Biedl,
\emph{Small drawings of outerplanar graphs, series-parallel graphs,
and other planar graphs},
Discrete \& Computational Geometry \textbf{45} (2011), 141--160.
\href{https://doi.org/10.1007/s00454-010-9310-z}
{doi:10.1007/s00454-010-9310-z}.

\bibitem{BFL2021}
V.~Bla\v{z}ej, J.~Fiala and G.~Liotta,
\emph{On edge-length ratios of partial $2$-trees},
International Journal of Computational Geometry \& Applications
\textbf{31} (2021), no.~2--3, 141--162.
\href{https://doi.org/10.1142/S0218195921500072}
{doi:10.1142/S0218195921500072}.
Conference-version preprint, \emph{On the edge-length ratio of $2$-trees},
\href{https://arxiv.org/abs/1909.11152v3}{arXiv:1909.11152v3}
(2020), Theorem~1 and Corollary~3.

\bibitem{BorrazzoFrati2020}
M.~Borrazzo and F.~Frati,
\emph{On the planar edge-length ratio of planar graphs},
Journal of Computational Geometry \textbf{11} (2020), no.~1, 137--155.
\href{https://doi.org/10.20382/jocg.v11i1a6}
{doi:10.20382/jocg.v11i1a6}.

\bibitem{Chen2011}
J.~Chen, A.~Jiang, I.~A. Kanj, G.~Xia and F.~Zhang,
\emph{Separability and topology control of quasi unit disk graphs},
Wireless Networks \textbf{17}(1) (2011), 53--67.
\href{https://doi.org/10.1007/s11276-010-0264-0}
{doi:10.1007/s11276-010-0264-0}.

\bibitem{DiBattista1999}
G.~Di~Battista, P.~Eades, R.~Tamassia and I.~G. Tollis,
\emph{Graph Drawing: Algorithms for the Visualization of Graphs},
Prentice-Hall, 1999.

\bibitem{EadesWormald1990}
P.~Eades and N.~C. Wormald,
\emph{Fixed edge-length graph drawing is {NP}-hard},
Discrete Applied Mathematics \textbf{28}(2) (1990), 111--134.
\href{https://doi.org/10.1016/0166-218X(90)90110-X}
{doi:10.1016/0166-218X(90)90110-X}.

\bibitem{LazardLenhartLiotta2019}
S.~Lazard, W.~J. Lenhart and G.~Liotta,
\emph{On the edge-length ratio of outerplanar graphs},
Theoretical Computer Science \textbf{770} (2019), 88--94.
\href{https://doi.org/10.1016/j.tcs.2018.10.002}
{doi:10.1016/j.tcs.2018.10.002}.

\bibitem{PachToth2004}
J.~Pach and G.~T\'oth,
\emph{Monotone drawings of planar graphs},
Journal of Graph Theory \textbf{46} (2004), no.~1, 39--47.
\href{https://doi.org/10.1002/jgt.10168}{doi:10.1002/jgt.10168}.
Revised version \href{https://arxiv.org/abs/1101.0967}
{arXiv:1101.0967} (2011), Theorem~2.

\end{thebibliography}
\end{document}